%% file: main.tex
\documentclass[reqno, 11pt]{amsart}

\input{Preamble}

\input{specificPreamble}

\definecolor{cMaroon}{HTML}{93152a}
\newcommand{\defn}[1]{{\color{cMaroon}{\emph{#1}}}}
\newcommand{\defnm}[1]{{\color{cMaroon}{{#1}}}}

\title[Counterexamples regarding linked and lean tree-decompositions]{Counterexamples regarding linked and lean tree-decompositions of infinite graphs}

\author{Sandra Albrechtsen}

\author{Raphael~W.\ Jacobs}

\author{Paul Knappe}

\author{Max Pitz}

\address{University of Hamburg, Department of Mathematics, Bundesstraße 55 (Geomatikum), 20146 Hamburg, Germany}
\email{\{sandra.albrechtsen, raphael.jacobs, paul.knappe, max.pitz\}@uni-hamburg.de}


\begin{document}

\begin{abstract}
    K{\v r}{\'i}{\v z} and Thomas showed that every (finite or infinite) graph of tree-width $k \in \N$ admits a lean \td\ of width $k$.
    We discuss a number of counterexamples demonstrating the limits of possible generalisations of their result to arbitrary infinite tree-width.

    In particular, we construct a locally finite, planar, connected graph that has no lean \td.
\end{abstract}

\maketitle

\section{Introduction} \label{sec:Intro}

\subsection{Lean \td s}
A cornerstone in both Robertson and Seymour's work \cite{GMIV} on well-quasi-ordering finite graphs, and in Thomas's result \cite{thomas1989wqo} that the class of infinite graphs of tree-width~$< k$ is well-quasi-ordered under the minor relation for all~$k \in \N$, is K{\v r}{\'i}{\v z} and Thomas's result on lean \td s. 
Recall that a \td\ $(T, (V_t)_{t \in T})$ is \defn{lean} if for every two (not necessarily distinct) nodes $s,t \in T$ and  sets of vertices $Z_s \subseteq V_{s}$ and $Z_t \subseteq V_{t}$ with $|Z_s| = |Z_t| =: \ell \in \N$, either $G$ contains $\ell$ pairwise disjoint $Z_s$--$Z_t$ paths or there exists an edge $e =xy \in sTt$ whose corresponding adhesion set $V_e := V_x \cap V_y$ has size less than $\ell$.

\begin{theorem}[{Thomas 1990 \cite{LeanTreeDecompThomas}, K{\v r}{\'i}{\v z} and Thomas 1991  \cite{kriz1991mengerlikepropertytreewidth}}]
\label{thm_intro_krizthomas}
    For every $k \in \N$, every (finite or infinite) graph  of tree-width  $< k$ has a lean tree-decomposition of width $< k$.
\end{theorem}

Is it possible to generalise \cref{thm_intro_krizthomas} from finite $k$ to arbitrary infinite cardinalities?
In the following let $\kappa$ be an infinite cardinal. A graph $G$ has \defn{tree-width $<\kappa$} if it admits a tree-decomposition of width $< \kappa$, i.e. one into parts of size $<\kappa$. A graph $G$ of tree-width $<\aleph_0$, i.e.\ with a tree-decomposition into finite parts, is said to have \defn{finite tree-width}.
The following questions arise naturally:

\begin{enumerate}
    \item \label{item:leantdintopartsofsizekappa} Does every graph of tree-width  $<\kappa$ admit a lean \td\ of width $<\kappa$? 
    In particular, does every graph of finite tree-width admit a lean \td\ into finite parts?
    \item \label{item:leantdintonotnecessarilyfiniteparts} If not, does every infinite graph at least admit a lean \td?
\end{enumerate}

\noindent Note that even without the width restriction, Question \ref{item:leantdintonotnecessarilyfiniteparts} remains non-trivial as the leanness-property has to be satisfied within each bag, which means that we cannot take the trivial \td\ into a single part, unless the graph is infinitely connected. Still, our main example shows that the answers to these questions are in the negative: 

\begin{mainexample} \label{main:NoLeanTD}
    There is a planar, locally finite, connected graph that admits no lean \td. 
\end{mainexample}

Every locally finite, connected graph is countable, and thus has tree-width $<\aleph_0$: given an arbitrary enumeration $\{v_i \colon i \in \N\}$ of a countable graph $G$, assigning to each vertex $r_i$ of a ray $R = r_0 r_1 \dots$ the bag $V_{r_i} := \{v_0, \dots, v_i\}$ yields a ray- and thus also \td\ $(R,\cV)$ of $G$ into finite parts.
Hence, the graph from \cref{main:NoLeanTD} witnesses that the answers to both questions \cref{item:leantdintopartsofsizekappa} and \cref{item:leantdintonotnecessarilyfiniteparts} are in the negative. 

On the positive side, we provide in \cite{LinkedTDInfGraphs}*{Theorem~3} a sufficient criterion that guarantees the existence of a lean \td\ into finite parts:

\begin{theorem}
\label{thm_LeanTDWithoutHalfGrid}
    Every graph without half-grid minor admits a lean \td\ into finite parts.
\end{theorem}

 Note that excluding the half-grid as a minor is sufficient but not necessary for the existence of lean \td s into finite parts: The countably infinite clique $K^{\aleph_0}$ contains the half-grid even as a subgraph but also the above described \rd\ of any given countable graph~$G$ into finite parts is lean for $G = K^{\aleph_0}$.

As the graph from \cref{main:NoLeanTD} is planar, it has no $K^5$ minor, and thus no $K^{\aleph_0}$ minor. 
Hence, in terms of excluded minors, the gap between our positive result \cref{thm_LeanTDWithoutHalfGrid} and our negative result \cref{main:NoLeanTD} is quite narrow.
Nevertheless, it remains open to exactly characterise the graphs which admit a lean \td\ (into finite parts, or more generally, of width $<\kappa$).

\subsection{Linked \td s}

Since the answers to questions \cref{item:leantdintopartsofsizekappa} and \cref{item:leantdintonotnecessarilyfiniteparts} are in the negative, it is natural to ask what happens if we weaken the condition that the \td\ be lean.
One possible such weakening is suggested by the `linkedness'-property, which was extensively studied in \cites{erde2018unified,bellenbaum2002two}: We say a \td\ $(T, \cV)$ of a graph $G$ is 
\begin{itemize}
    \item \defn{strongly linked} if for every two nodes $s \neq t$ of $T$ there are $\min\{|V_e| \colon e \in E(sTt)\}$ pairwise disjoint $V_s$--$V_{t}$ paths in $G$.
\end{itemize}

\noindent One may further weaken `strongly linked' by requiring the existence of disjoint $V_s$--$V_t$ paths only between nodes $s, t \in T$ that are `comparable'.
For this, recall that given a tree $T$ rooted at a node~$r$, its \defn{tree-order} is given by $s \leq t$ for nodes $s, t$ of $T$ if~$s$ lies on the (unique) path \defn{$rTt$} from $r$ to~$t$.
Thomas~\cite{thomas1989wqo} defined a rooted \td\ $(T, \cV)$ of a graph $G$ to be 

\begin{itemize}
    \item \defn{linked} if for every two comparable nodes $s < t$ of the rooted tree $T$ there are $\min\{|V_e| : e \in E(sTt)\}$ pairwise disjoint $V_s$--$V_{t}$ paths in $G$.
\end{itemize}

\noindent These weakenings of the `leanness'-property are motivated by the fact that for the aforementioned applications of \cref{thm_intro_krizthomas} by Robertson and Seymour \cite{GMIV}  and by Thomas \cite{thomas1989wqo} it is only important that the rooted \td\ is linked.
\medskip

So what happens if we replace the condition \emph{lean} in \cref{item:leantdintopartsofsizekappa} and \cref{item:leantdintonotnecessarilyfiniteparts} by \emph{strongly linked} or \emph{linked}?
As the trivial \td\ into a single part is always strongly linked, and thus every graph has a strongly linked \td, only the weakened version of \cref{item:leantdintopartsofsizekappa}, but not of \cref{item:leantdintonotnecessarilyfiniteparts}, is interesting:

\begin{enumerate}
    \setcounter{enumi}{2}
    \item \label{item:stronglylinked} Does every graph of tree-width $< \kappa$ admit a strongly linked \td\ of width~$<\kappa$?
\end{enumerate}

However, question \cref{item:stronglylinked} is trivially true for all infinite cardinalities $\kappa$: 
By definition, every graph $G$ of tree-width $<\kappa$ admits a \td\ $(T, \cV)$ of width $<\kappa$. Choose an arbitrary root $r$ of $T$.
By assigning to each of the nodes~$t$ of $T$ the bag $V'_t := \bigcup_{s \in rTt} V_s$ we obtain a (rooted) \td\ $(T,\cV')$ of width $<\kappa$ that is strongly linked; albeit for the trivial reason that~$s \neq t \in T$ implies~$V'_s \cap V'_t = V'_u$ where~$u$ is the $\leq$-minimal node in $sTt$.

But this \td\ is not useful in practice, and so one would like to have some additional properties making the \td\ less redundant.
Especially linked rooted \td s into finite parts which are additionally `tight' and `componental' turned out to be a powerful tool (cf.\ \cite{LinkedTDInfGraphs}*{\S1.2-1.4} and \cite{SATangleTreeDualityInfGraphs}; see paragraph after \cref{thm:LinkedTightCompTreeDecomp} below for more details).
Given a rooted \td\ $(T, \cV)$ of a graph~$G$ and an edge $e$ of $T$, we call the subgraph \defn{$G \strictup e$} of~$G$ induced on $\bigcup_{t \in T_e} V_t \setminus V_e$ the \defn{part strictly above $e$}, where $T_e$ is the unique component of $T-e$ that does not contain the root of $T$. Then $(T, \cV)$ is

\begin{itemize}
    \item \defn{componental} if all the parts $G \strictup e$ strictly above edges $e \in E(T)$ are connected, and
    \item \defn{tight} if for every edge $e \in E(T)$ there is some component $C$ of $G \strictup e$ such that $N_G(C) = V_e$.
\end{itemize}

\noindent We remark that the above strongly linked (rooted) \td\ $(T, \cV')$ is componental if $(T, \cV)$ was componental; but even if $(T, \cV)$ was tight, $(T, \cV')$ may no longer be tight.

The property \emph{tight} ensures that the adhesion sets contain no `unnecessary' vertices. Any given componental rooted \td\ can easily be transformed into a tight and componental rooted \td\ by deleting for every edge $e$ of $T$ the non-neighbours of $G \strictup e$ in $V_e$ from every $V_t$ with $t \in T_e$.
While this construction obviously does not increase the width, it does not necessarily maintain the property \emph{(strongly) linked}. 
So it is natural to strengthen \cref{item:stronglylinked} to ask whether there exists a \td\ which has all three properties:

\begin{enumerate}
    \setcounter{enumi}{3}
    \item \label{item:stronglylinkedPlusTC} Does every graph of tree-width $<\kappa$ admit a tight and componental rooted \td\ of width $<\kappa$ that is strongly linked?
    \item \label{item:linkedPlusTC} If not, does every graph of tree-width $<\kappa$ admit a tight and componental rooted \td\ of width $<\kappa$ that is at least linked?
\end{enumerate}

The answer to Question~\ref{item:stronglylinkedPlusTC} is in the negative already for $\kappa = \aleph_0$, as witnessed by the same graph that we constructed for \cref{main:NoLeanTD}:

\begin{mainexample} \label{main:NoUnrootedLinkedTD}
There is a planar, locally finite, connected graph which admits no tight, componental rooted \td\ into finite parts that is strongly linked. 
\end{mainexample}

Question~\ref{item:linkedPlusTC}, however, has an affirmative answer for $\kappa = \aleph_0$ \cite{LinkedTDInfGraphs}*{Theorem 1}:

\begin{theorem} \label{thm:LinkedTightCompTreeDecomp}
    Every graph of finite tree-width admits a rooted \td\ into finite parts that is linked, tight and componental.
\end{theorem}

We remark that \cref{item:linkedPlusTC} remains open for uncountable cardinalities $\kappa > \aleph_0$. If the answer is positive, one might, as a next step, also strengthen the notion of \emph{linked} from just considering sizes to a structural notion that encapsulates the typical desired behaviour of infinite path families between two sets, as it is given by Menger's theorem for infinite graphs \cite{infinitemenger}*{Theorem~1.6} proven by Aharoni and Berger.
\medskip

It turns out that linked rooted \td s into finite parts which are additionally tight and componental, as given by \cref{thm:LinkedTightCompTreeDecomp}, are particularly useful (cf.\ \cite{LinkedTDInfGraphs}*{\S1.2-1.4} and \cite{SATangleTreeDualityInfGraphs}):
In \cite{LinkedTDInfGraphs}*{\S3}, we show that rooted \td s into finite parts which are linked, tight and componental display the end structure of the underlying graph.
This not only resolves a question of Halin \cite{halin1977systeme}*{\S6} but also allows us to deduce from \cref{thm:LinkedTightCompTreeDecomp}, by means of short and unified proofs, the characterisations due to Robertson, Seymour
and Thomas of graphs without half-grid minor \cite{robertson1995excluding}*{Theorem 2.6}, and of graphs without binary tree subdivision \cite{seymour1993binarytree}*{(1.5)}.
Also the proof of \cref{thm_LeanTDWithoutHalfGrid} in \cite{LinkedTDInfGraphs}*{\S8} heavily relies on post-processing the \td\ from \cref{thm:LinkedTightCompTreeDecomp}.
Beside these, there are more applications of rooted \td\ into finite parts which are linked, tight and componental in \cite{LinkedTDInfGraphs}*{\S1.4} and \cite{SATangleTreeDualityInfGraphs}.

In fact, we show in \cite{LinkedTDInfGraphs}*{Theorem~1'} a more detailed version of \cref{thm:LinkedTightCompTreeDecomp} which, among others, yields that the adhesion sets of the \td\ intersect `not more than necessary'. We also give an example which proves that this property is best possible even for locally finite graphs (see \cref{sec:IncDisj} for details).
\medskip

In the light of Question~\cref{item:linkedPlusTC} being true for $\kappa = \aleph_0$, one may ask whether a similar modification could rescue \cref{item:leantdintopartsofsizekappa} for $\kappa = \aleph_0$: What happens if we relax the condition \emph{lean} in \cref{item:leantdintopartsofsizekappa}  to a corresponding `rooted' version?

\begin{enumerate}
    \setcounter{enumi}{5}
    \item \label{item:RootedlyLeanTD} Does every graph of finite tree-width admit a rooted \td\ into finite parts that satisfies the property of being lean for all \emph{comparable} nodes $s \leq t$ of $T$?
\end{enumerate}

\noindent However, the answer to this question is again in the negative, as there is a graph of finite tree-width such that all its \td s into finite parts violate  the `leanness'-property within a single bag:

\begin{mainexample} \label{main:NoRootedlyLeanTreeDecomp}
    There is a countable graph $G$ such that every \td\ of $G$ into finite parts has a bag $V_t$ which violates the property of being lean for $s = t$.
\end{mainexample}

We remark that we do not know whether every \td\ which satisfies the `leanness'-property for every two comparable nodes must already be lean.

\subsection{How this paper is organised}

In \cref{sec:Prelims} we recall some important definitions and facts about ends and their interplay with \td s. In \cref{sec:NoLeanTD} we prove \cref{main:NoLeanTD,main:NoUnrootedLinkedTD}, and in \cref{sec:NoRootedlyLeanTD} we prove \cref{main:NoRootedlyLeanTreeDecomp}. Finally, in \cref{sec:IncDisj}, we discuss whether it is possible to strengthen \cref{thm:LinkedTightCompTreeDecomp} so that the adhesion sets of the \td\ are `upwards' disjoint.

\section{Preliminaries} \label{sec:Prelims}

We refer the reader to \cite{LinkedTDInfGraphs}*{\S2} for all the relevant definitions.
Besides that, our notation and definitions follow \cite{bibel}.
In particular, the \defn{neighbourhood $N_G(X)$} of a set $X$ of vertices in a graph $G$ is the set of all vertices in $V(G) \setminus X$ that are neighbours of some $x \in X$.
For convenience of the reader we recall the ones which are most important in this paper.

Let $(T, \cV)$ be a rooted \td\ of a graph~$G$.
Given an edge~$e = t_1 t_2$ of~$T$ with~$t_1 <_T t_2$, we let~$T_i$ be the component of $T-e$ containing~$t_i$ for $i =1,2$. 
We abbreviate the sides of its induced separation by~$G \down e := G \left[ \bigcup_{t \in T_1} V_t \right]$ and~$G \up e := G \left[ \bigcup_{t \in T_2} V_t \right]$.
Further, we write~$G \strictdown e := G \down e - V_e$ and~$G \strictup e := G \up e - V_e$.
It is easy to see that $G \up e \supseteq G \up f$, $G \strictup e \supseteq G \strictup f$, $G \down e \subseteq G \down f$ and $G \strictdown e \subseteq G \strictdown f$ for every two edges $e, f \in E(T)$ with $e \leq_T f$, and also~$\bigcap_{e \in R} G\strictup e = \emptyset$ for every rooted ray~$R$ in~$T$.

An \defn{end} $\eps$ of a graph $G$ is an equivalence class of rays in $G$ where two rays are equivalent if for every finite set $X$ they have a tail in the same component of~$G-X$. We refer to this component of $G-X$ as $C_G(X, \eps)$.
Write \defn{$\Omega(G)$} for the \defn{set of all ends} of $G$.
Note that for every end $\eps$ of $G$ and every set $X$ of vertices of $G$ which meets every $\eps$-ray at most finitely often, there is a unique component of $G-X$ that contains a tail of every $\eps$-ray. We refer to this component as \defn{$C_G(X, \eps)$}, which coincides with above definition for finite $X$.

A vertex $v$ of $G$ \defn{dominates} an end $\eps$ of $G$ if it lies in $C_G(X, \eps)$ for every finite set $X$ of vertices other than $v$.
We remark that locally finite graphs have no dominating vertices.
We denote the set of all vertices of $G$ which dominate an end $\eps$ by \defn{$\Dom(\eps)$}.
The \defn{degree $\deg(\eps)$} of an end $\eps$ of $G$ is the supremum over all cardinals $\kappa$ such that there exists a set of $\kappa$ pairwise disjoint rays in $\eps$, and its \defn{combined degree} is~$\defnm{\Delta(\eps)} := \deg(\eps) +|\Dom(\eps)|$.

Given a rooted \td\ $(T, \cV)$ of $G$, an end $\eps$ of $G$ \defn{gives rise} to a rooted ray $R$ in $T$ if every $\eps$-ray meets every bag $V_t$ with $t \in R$ at most finitely often, and for every $e = st \in T$ with $s < t$ we have $C_G(V_s, \eps) \subseteq G \strictup e$\footnote{We remark that this definition extends the one given in \cite{LinkedTDInfGraphs}*{\S2.2} to a larger class of \td s which do not necessarily have finite adhesion.}.
Note that if such a ray exists, then it is unique, and we denote it by $R_\eps$.
Moreover, if every ray in an end $\eps$ of $G$ meets every bag of a \td\ $(T, \cV)$ at most finitely often, then $\eps$ gives rise to a ray in $T$.
If the bags of a rooted \td\ $(T, \cV)$ meet all rays in $G$ at most finitely often, the above yields a map~$\phi \colon \Omega(G) \to \Omega(T),\; \eps \mapsto R_\eps$. 
Such a rooted \td~$(T, \cV)$
\begin{itemize}
\item  \defn{displays} the ends of $G$ if $\phi$ is a bijection  \cite{carmesin2019displayingtopends},
\item \defn{displays all dominating vertices} if $\liminf_{e \in E(R_\eps)} V_e = \Dom(\eps)$\footnote{The set-theoretic $\liminf_{n \in \N} A_n$ consists of all points that are contained in all but finitely many $A_n$. For a ray $R= v_0e_0v_1e_1v_1 \dots$ in $T$, one gets $\liminf_{e \in E(R_\eps)} V_e = \bigcup_{n \in \N} \bigcap_{i \geq n} V_{e_i}$.} for all $\eps \in \Omega(G)$, and
\item \defn{displays all combined degrees} if $\liminf_{e \in E(R_\eps)} |V_e| = \Delta(\eps)$ for all $\eps \in \Omega(G)$. 
\end{itemize}

Moreover, we recall the following lemma \cite{LinkedTDInfGraphs}*{Lemma~3.3}:

\begin{lemma} \label{lem:linkedandliminfareonlydominatingverticesimpliesdisplayenddegrees}
    Let $(T, \cV)$ be a linked rooted \td\ of a graph~$G$ which has finite adhesion. Suppose that an end $\eps$ of $G$ gives rise to a ray in $T$ which arises from no other end of~$G$ and that $\liminf_{e \in R} V_e = \Dom(\eps)$.
	Then $\liminf_{e \in R} |V_e| = \Delta(\eps)$.
\end{lemma}

The following lemma follows immediately from \cite{LinkedTDInfGraphs}*{Lemmas~3.1, 3.2 \& 3.3}:

\begin{lemma} \label{lem:LinkedTightCompImpliesDisplayingEndStructure}
    Let $(T, \cV)$ be a rooted \td\ of a graph $G$ which has finite adhesion and which is linked, tight and componental. Then $(T, \cV)$ displays all ends of $G$, their dominating vertices and their combined degrees. \qed
\end{lemma}

\section{\texorpdfstring{\cref{main:NoLeanTD,main:NoUnrootedLinkedTD}}{Examples 1 and 2} -- negative answers to questions \ref{item:leantdintopartsofsizekappa},\ref{item:leantdintonotnecessarilyfiniteparts} and \ref{item:stronglylinkedPlusTC}} \label{sec:NoLeanTD}

In this section we explain \cref{main:NoLeanTD,main:NoUnrootedLinkedTD}, which we restate here for convenience:

\begin{customex}{\cref{main:NoLeanTD}}\label{main:NoLeanTD:copy}
    There is a planar, locally finite, connected graph that admits no lean \td.
\end{customex}

\begin{customex}{\cref{main:NoUnrootedLinkedTD}}
\label{main:NoUnrootedLinkedTD:cop}
    There is a planar, locally finite, connected graph which admits no tight, componental rooted \td\ into finite parts that is strongly linked.
\end{customex}

\begin{figure}[h!t]
    \centering
    \pdfOrNot{\includegraphics[page=2]{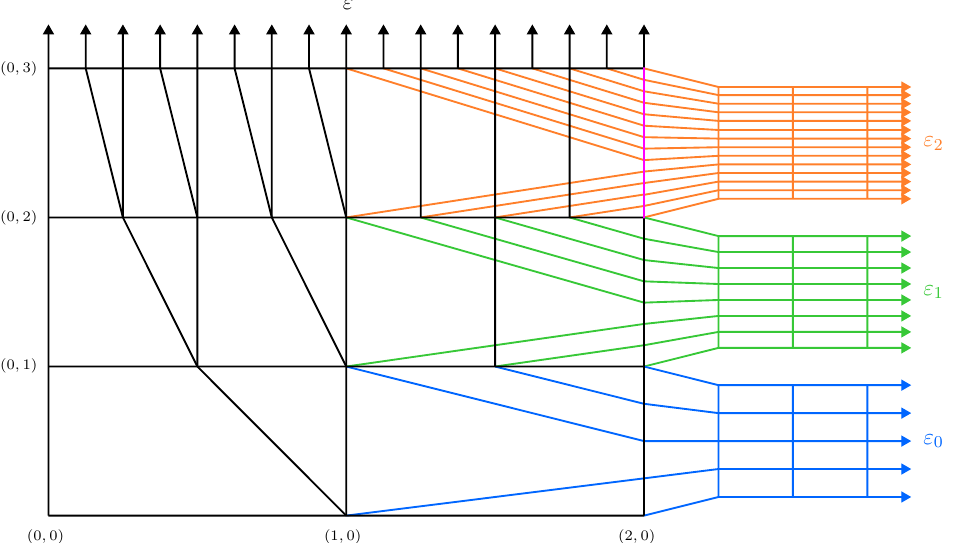}}{\includesvg[width=1.03\columnwidth]{svg/planarexample.svg}}
    \caption{Depicted is the graph $G$ from \cref{const:THEGraph}. The subgraph induced by the orange edges is $G_2$ together with the extensions of its horizontal rays.}
    \label{fig:THEGraph}
\end{figure}

For our proofs of \cref{main:NoLeanTD,main:NoUnrootedLinkedTD} we construct a graph $G$ in \cref{const:THEGraph} below, and then show that $G$ is already as desired for both \cref{main:NoLeanTD,main:NoUnrootedLinkedTD}.
The graph in this construction is inspired by~\cite{carmesin2022canonical}*{Example~7.4}.\footnote{The example is only presented in the arXiv version of \cite{carmesin2022canonical}.}

\begin{construction} \label{const:THEGraph}
Let $G$ be the graph depicted in \cref{fig:THEGraph}. Formally, let $G'$ be the graph on the vertex set $V(G') := \{(i/2^j, j) \mid j \in \N, 0 \leq i \leq 2^{j+1}\}$ and with edges between $(i/2^j, j)$ and $((i+1)/2^j, j)$, between $(i/2^j, j)$ and $(i/2^j, j+1)$, and also between $(i/2^j, j)$ and $((2i-1)/2^{j+1}, j+1)$ for $i \leq 2^j$ (this is the black subgraph in \cref{fig:THEGraph}). 
Note that $G'$ has a unique end, which we denote by $\eps$.

For every $n \in \N_{\geq 1}$, let $G_n := \N_{\geq 1} \boxtimes P_{2^{n+1}
 + 2^n + 1}$ be a grid with $2^{n+1} + 2^n + 2$ rows and infinitely many columns. 
Then $G_n$ is one-ended and its end $\eps_n$ has degree $2^{n+1}  + 2^n + 2$.
Now the graph $G$ is obtained from $G' \sqcup \bigsqcup_{n \in \N} G_n$ by deleting for every $n \in \N$ the edge $\{(2,n), (2, n+1)\}$ of $G'$, identifying the vertices $(2,n), (2,n+1)$ of $G'$ with the respective vertices $(1,1)$ and $(1,2^{n+1} + 2^n + 2)$ of $G_n$, and by extending the horizontal rays in $G_n$ as indicated in \cref{fig:THEGraph}. In particular, the extended horizontal rays of each $G_n$ are still disjoint and their initial vertices are precisely $U_n := \{(i/2^j, j) \mid j \in \{n,n+1\}, 2^j \leq i \leq 2^{j+1}\}$.
To get a graph which is not only locally finite but also planar, we subdivide the edges between $(i/2^{n}, n)$ and $(i/2^{n}, n+1)$ with $i > 2^{n}$ in~$G'$ to obtain the extended rays. 
For later use, we refer to the vertex set $\{(i,j) \mid 1 \leq j \leq 2^{n+1} + 2^n + 2\}$ in $G_n$ as the \defn{$i$-th column of $G_n$} (indicated in pink in \cref{fig:THEGraph} is the first column of $G_2$) and also set
\[
\defnm{S_n} := \{(i/2^{n+1}, n+1) \mid 2^{n+1} < i \leq 2^{n+2}\} \cup \{(1,j) \mid 0 \leq j \leq n+1\}
\]
for all $n \in \N$ (indicated in purple in \cref{fig:CounterexUnrootedLinked2}.)
This completes the construction.
\end{construction}

In the remainder of this section we prove that the graph $G$ from \cref{const:THEGraph} is as desired for \cref{main:NoLeanTD,main:NoUnrootedLinkedTD}. 
For this, we first show two auxiliary lemmas. The first says that $G$ does not admit a \td\ which
`efficiently distinguishes' all ends of $G$. 
Recall that in a \td\ $(T, \cV)$ of $G$ every edge $e = t_0t_1$ of $T$ induces a separation as follows: For $i = 0,1$ write $T_i$ for the component of $T-e$ that contains $t_i$. Then $\{\bigcup_{s \in T_0} V_s, \bigcup_{s \in T_1} V_s\}$ is a separation of $G$ \cite{bibel}*{Lemma~12.3.1}. 
A separation $\{A,B\}$ of $G$ \defn{efficiently distinguishes} two ends $\eps, \eps'$ of $G$ if $\eps$ and $\eps'$ live in components of $G - (A \cap B)$ on different sides of $\{A,B\}$, and there is no separation $\{C,D\}$ of $G$ of smaller order $|C \cap D|$ with this property.
A \td\ $(T, \cV)$ \defn{distinguishes} two ends $\eps, \eps'$ of $G$ \defn{efficiently} if some edge $e$ of $T$ induces a separation which efficiently distinguishes $\eps$ and $\eps'$. 

\begin{lemma} \label{lem:NoTreeDecompEffDistAllEnds}
    Let $(T, \cV)$ be a \td\ of the graph from \cref{const:THEGraph}. Then there exist $n \in \N$ such that
    $(T, \cV)$ does not efficiently distinguish $\eps_n$ and $\eps$.
\end{lemma}

\begin{figure}[ht]
    \centering
    \pdfOrNot{\includegraphics[page=2]{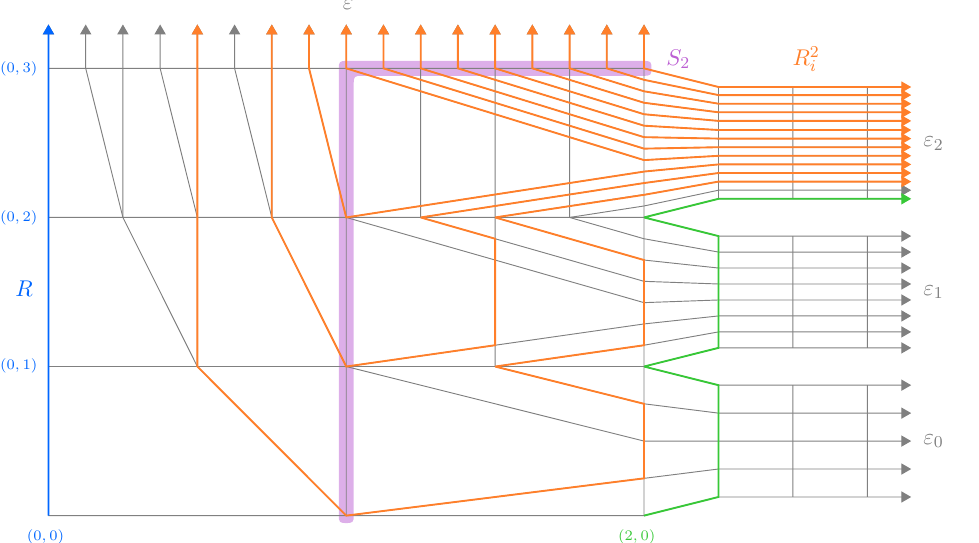}}{\includesvg[width=1.03\columnwidth]{svg/planarexample_rays.svg}}
    \caption{Depicted is the situation in the proof of \cref{lem:NoTreeDecompEffDistAllEnds} for $n = 2$.}
    \label{fig:CounterexUnrootedLinked2}
\end{figure}

\begin{proof}
    Suppose towards a contradiction that $G$ admits a \td\ $(T, \cV)$ such that, for every end $\eps_n$ of $G$, there exists an edge $f_n$ such that the separation induced by $f_n$ distinguishes $\eps_n$ and $\eps$ efficiently. Then $V_{f_n}$ has size at most $2^{n+1}+n+2$ as witnessed by $S_n$ (indicated in purple in \cref{fig:CounterexUnrootedLinked2}). By the definition of~$G'$, there are in fact $|S_n|$ disjoint $\eps$--$\eps_n$ double rays $R_i^n$ in~$G$ (indicated in orange in \cref{fig:CounterexUnrootedLinked2}), so every sizewise-minimal $\eps$--$\eps_n$ separator, and in particular~$V_{f_n}$, has to meet each $R_i^n$ precisely once. In particular, $|V_{f_n}| = |S_n|$ and hence $f_n \neq f_m$ for all $n \neq m \in \N$. Moreover, $V_{f_n}$ avoids the $\eps$-ray $R = (0,0) (0,1) \dots$ (indicated in blue in \cref{fig:CounterexUnrootedLinked2}) and the vertex $(2,0)$, as they are both disjoint from all the $\eps$--$\eps_n$ double rays $R_i^n$. Thus, $G-V_{f_n}$ has a component $C_n$ that contains $R$, and a component $C'_n$ that contains $(2,0)$. In particular, $\eps$ lives in $C_n$. Since there is an $\eps_n$-ray (indicated in green in \cref{fig:CounterexUnrootedLinked2}) which starts in $(2,0)$ and is disjoint from the $R_i^n$, $\eps_n$ lives in $C'_n$. As~$V_{f_n}$ separates $\eps$ and $\eps_n$, we have $C_n \neq C'_n$. Thus, $(0,0)$ and $(2,0)$ lie on different sides of the separation induced by $f_n$. 
    
    Let $V_t$ and $V_s$ be some bags of $(T, \cV)$ which contain $(0,0)$ and $(2, 0)$, respectively. Since the separations induced by every $f_n$ separate $(0,0)$ and $(2,0)$, all the infinitely many distinct edges $f_n$ lie on the finite path $tTs$, which is a contradiction.
\end{proof}

The next lemma essentially says that every \td\ of~$G$ that displays all ends of~$G$ and their combined degrees cannot be strongly linked.

\begin{lemma} \label{lem:NoStronglyLinkedTDDisplayingEndDegrees}
    Let $(T, \cV)$ be a rooted \td\ of the graph $G$ from \cref{const:THEGraph}. Suppose that every end $\omega$ of $G$ gives rise to a rooted ray $R_\omega$ in $T$ with $\liminf_{e \in R_\omega} V_e = \emptyset$ and $\liminf_{e \in R_\omega} |V_e| = \deg(\omega)$. Then $(T, \cV)$ is not strongly linked.
\end{lemma}

\begin{proof}
    Suppose towards a contradiction that $(T, \cV)$ is strongly linked. Let $n \in \N$ be arbitrary.
    We write $R_n$ for the rooted ray $R_{\eps_n}$ in $T$, which arose from $\epsilon_n$.
    Since $(T,\cV)$ is a \td , for every ray $R$ of $T$, we have $\bigcap_{e \in R} V(G\strictup e) = \emptyset$, and thus for every finite vertex set $W$ of $G$, all but finitely many edges $e$ of $R$ have the property that $G\strictup e$ avoids $W$. In particular, for all but finitely many edges $e$ of $R_n$, the part $G \strictup e$ avoids the first and second column of $G_n$. 
    As $\liminf_{e \in R_n} |V_e| = \deg(\epsilon_n)$, we may choose such an edge $e_n \in R_n$ so that $|V_{e_n}| = \deg(\eps_n)$ and every later edge $e$ on $R_n$ satisfies $|V_e| \geq |V_{e_n}|$.
    Since $G \strictup e_n$ avoids the first column of $G_n$, the component $D_n$ of $G \strictup e_n$ in which $\eps_n$ lives is contained in $G_n$. 
    
    We claim that $G \strictup e_n$ contains all but finitely many vertices of $G_n$, and $N_G(D_n) = V_{e_n}$. 
    Indeed, the rows of $G_n$ are pairwise disjoint $\eps_n$-rays.
    Since $\eps_n$ lives in $D_n$, the component $D_n$ contains a tail of each row of $G_n$.
    As the rows of $G_n$ cover the entire $G_n$, the part $G \strictup e_n \supseteq D_n$ contains all but finitely many vertices of $G_n$. 
    Since $D_n \subseteq G \strictup e_n$ avoids the first column of $G_n$, its neighbourhood $N_G(D_n) \subseteq V_{e_n}$ contains at least one vertex of each row. 
    Since there are $\deg(\eps_n)$ rows, we have $|N_G(D_n)| \geq \deg(\eps_n) = |V_{e_n}|$, and thus $N_G(D_n) = V_{e_n}$.

    As $N_G(D_n) = V_{e_n}$ and $D_n$ is connected,
    $G \up e_n$ is connected. It follows that $G \up e_n$ is a subgraph of $G_n$ and avoids the first column of~$G_n$, since $G \strictup e_n$ meets $G_n$ and avoids the first and second column of $G_n$.

    Let $H_n$ be the component of $G-S_n$ that contains $G_0$ (where $S_n$ is the set indicated in purple in \cref{fig:CounterexUnrootedLinked2}). We claim that there is an edge~$e'_n$ of the rooted ray~$R_\eps$ in~$T$, which arose from~$\epsilon$, such that $|V_{e'_n}| \geq \deg(\eps_n)$ and such that $G \up e$ avoids $H_n$. 
    For this, we note that the rays $R_m$ are distinct from $R_\epsilon$, since $\liminf_{e \in R_m} |V_e| = \deg(\epsilon_m)$ is finite but $\liminf_{e \in R_\epsilon} |V_e| = \deg(\epsilon)$ is infinite.
    Hence, we may choose a tail $R^n_\eps \subseteq R_\eps$ of $R_\eps$ that avoids all rays $R_m$ for $m \leq n$. 
    Since all $R_m$ and also $R_\eps$ are rooted at the same vertex of $T$, every edge $e$ of $R^n_\eps$ satisfies $(G\up e) \cap (G \strictup e_m) = \emptyset$.
    As $G\strictup e_m$ contains all but finitely many vertices of $G_m$ for $m \leq n$, the set $W_n := V(H_n) \setminus (\bigcup_{m \leq n} V(G\strictup e_n))$ is finite.
    Since $(T,\cV)$ is a \td , for all but finitely many edges $e$ of $R^n_\eps$, the part~$G\strictup e$ avoids $W_n$ and hence $H_n$.
    Note that, for each edge $e$ of $R^n_\eps$, its adhesion set $V_e$ avoids $G \strictup e_m$ for every $m \leq n$.
    Since $\liminf_{e \in R^n_\eps} V_e = \emptyset$ and $W_n$ is finite, for all but finitely many edges $e$ of~$R^n_\eps$, the adhesion set $V_e$ avoids the finite set $W_n$.
    Thus, for all but finitely many edges $e$ of $R_\eps$, the part $G \up e$ avoids $H_n$.
    Finally, since $\liminf_{e \in R_\epsilon} |V_e| = \deg(\epsilon)$ is infinite, we may choose an edge $e'_n$ of $R_\eps$ so that $G\up e'_n$ avoids $H_n$ and $|V_{e'_n}| \geq \deg(\eps_n)$.

    Recall that the adhesion set $V_{e_n}$ is contained in $G_n$ but avoids the first column of $G_n$, and thus $V_{e_n} \subseteq V(H_n)$.
    Since also $V_{e_n'} \cap V(H_n) = \emptyset$, there are at most $2^{n+1} + n + 2$ pairwise disjoint $V_{e_n}$--$V_{e_n'}$ paths in $G$, as witnessed by $S_n$. As $(T, \cV)$ is strongly linked by assumption, there is an edge $f_n$ on the unique $e_n$--$e_n'$ path in~$T$ such that $|V_{f_n}| \leq |S_n|$.
    Moreover, $V_{f_n}$ separates $V_{e_n}$ and $V_{e'_n}$, and thus also $\eps_n$ and $\eps$. 
    By the definition of $G'$, there are in fact $|S_n|$ disjoint $\eps$--$\eps_n$ double rays~$R^n_i$ in~$G$ (indicated in orange in \cref{fig:CounterexUnrootedLinked2}), and hence the separation induced by $f_n$ efficiently distinguishes $\eps$ and $\eps_n$. 
    Since $n \in \N$ was chosen arbitrarily, \cref{lem:NoTreeDecompEffDistAllEnds} yields the desired contradiction.
\end{proof}

\begin{proof}[Proof of \cref{main:NoUnrootedLinkedTD}]
    The graph $G$ from \cref{const:THEGraph} is planar, locally finite and connected. By \cref{lem:LinkedTightCompImpliesDisplayingEndStructure}, every linked, tight, componental rooted \td\ of~$G$ into finite parts displays all the ends of $G$, their dominating vertices and their (combined) degrees. Thus, \cref{lem:NoStronglyLinkedTDDisplayingEndDegrees} ensures that those \td s are not strongly linked. 
\end{proof}

We now turn to our proof of \cref{main:NoLeanTD}, i.e.\ that the graph $G$ from \cref{const:THEGraph} does not admit a lean \td. The proof consists of two steps. First, we show in \cref{lem:NoLeanTDIntoFiniteParts} that $G$ does not admit a lean \td\ into finite parts. Then, we show in \cref{lem:NoLeanTD} that $G$ neither admits a lean \td\ that has an infinite part.

By definition, every lean \td\ is in particular strongly linked. For the first step, it remains to show that every lean \td\ of $G$ into finite parts satisfies the premise of \cref{lem:NoStronglyLinkedTDDisplayingEndDegrees}: it displays all ends of $G$ and their combined degrees, up to the fact that maybe some rays of the decomposition tree do not arise from an end of $G$. 
In fact, the following \cref{lem:LeanTDKindOfDisplayEndsAndDegrees} together with \cref{lem:linkedandliminfareonlydominatingverticesimpliesdisplayenddegrees} implies that this holds true for \emph{all} graphs $H$ and lean \td s of them.

To prove \cref{lem:LeanTDKindOfDisplayEndsAndDegrees}, we first show that even though lean \td s may not be componental, they are not far away from it.

\begin{lemma} \label{lem:LeanTDsAreKindOfComponentalTight}
    Let $(T, \cV)$ be a rooted \td\ of finite adhesion of a graph $H$. Suppose there is an edge $e = st \in E(T)$ with $s <_{T} t$ and a set $Y \subseteq V_e$ such that $(H\up e) - Y$ has at least two components $C_1, C_2$. Suppose further that $V(C_1) \cap V_e \neq \emptyset$ and $V(C_2) \cap (V_t \setminus V_e) \neq \emptyset$. Then $(T, \cV)$ is not lean.
\end{lemma}

\begin{proof}
    By assumption, we may pick vertices $v_1 \in V(C_1) \cap V_e$ and $v_2 \in V(C_2) \cap (V_t \setminus V_e)$.
    Set $Z_1 := V_e$ and $Z_2 := (V_e\setminus \{v_1\}) \cup \{v_2\}$. By construction, $|Z_1| = |V_e| = |Z_2| =: k$. Since $(T, \cV)$ is lean and $Z_1, Z_2 \subseteq V_t$, there is a family $\cP$ of $k$ disjoint $Z_1$--$Z_2$ paths in $H$. Note that every path in $\cP$ that starts in $Z_1 \cap Z_2 = V_e\setminus \{v_1\}$ is a trivial path. Hence, there is a path $P \in \cP$ that starts in $\{v_1\} = Z_1 \setminus Z_2$ and ends in $\{v_2\} = Z_2 \setminus Z_1$. But $v_1 \in V(C_1)$ is separated from $v_2 \in V(C_2) \setminus V_e \subseteq V(H \strictup e)$ by $Y \cup (V(C_2) \cap V_e) \subseteq V_e \setminus \{v_1\}$. This contradicts that the paths in $\cP$ are pairwise disjoint.
\end{proof}

\begin{lemma} \label{lem:LeanTDKindOfDisplayEndsAndDegrees}
    Let $(T, \cV)$ be a rooted \td\ of a graph $H$ into finite parts which is lean as an unrooted \td. Then every ray in $T$ arises from at most one end of $H$. Moreover, if a ray $R$ in $T$ arises from an end $\eps$ of~$H$, then $\liminf_{e \in R} V_e = \Dom(\eps)$.
\end{lemma}

\begin{proof}
    For the first assertion, suppose towards a contradiction that there are two distinct ends $\eps_1, \eps_2$ of $H$ that give rise to the same rooted ray $R$ of~$T$. As $\eps_1, \eps_2$ are distinct, there is a finite set $X$ of vertices of $H$ such that $\eps_1, \eps_2$ live in distinct components $C_1, C_2$ of $G-X$.

    Now pick a vertex $v_1 \in V(C_1)$.
    Since $(T,\cV)$ is a \td, $\bigcap_{e \in R} V(H\strictup e) = \emptyset$, and thus all but finitely many edges $e$ of $R$ have the property that $H\down e$ contains the finite set $X \cup \{v_1\}$. As $\epsilon_2$ gives rise to $R$ and lives in $C_2$, there is an edge $e = st$ on $R$ with $s <_T t$ such that $X \cup \{v_1\} \subseteq H\down e$ and $V(C_2) \cap (V_t \setminus V_e)$ is non-empty. Note that $V(C_1) \cap V_e$ is non-empty as $C_1$ is connected and meets both $H\down e$ (in the vertex $v_1$) and $H \strictup e$ (as $\eps_1$ gives rise to $R$ and lives in $C_1$).
    Set $Y := X \cap V_e$ and, for $i=1,2$, let $C_i'$ be a component of $C_i \cap (H \up e - Y)$ which contains a vertex from $V(C_1) \cap V_e$ or $V(C_2) \cap (V_t \setminus V_e)$, respectively. Then $Y, C_1', C_2'$ are as in \cref{lem:LeanTDsAreKindOfComponentalTight}. It follows that $(T, \cV)$ is not lean, which is a contradiction.

    To show the moreover statement, let $\eps$ be an end of $H$ that gives rise to a ray $R$ in $T$. Now suppose towards a contradiction that there is a vertex $w \in \liminf_{e \in R} V_e$ that does not dominate~$\eps$. Then there is a finite set $X \subseteq V(H)$ and distinct components $C_1, C_2$ of $G-X$ such that $w \in V(C_1)$ and $\eps$ lives in $C_2$.

    As above, there is an edge $e = st$ on $R$ with $s <_T t$ such that $H \down e$ contains $X \cup \{w\}$ and $V(C_2) \cap (V_t \setminus V_e)$ is non-empty.
    Since $w$ also lies in $\liminf_{f \in R} V_f$, we have $w \in V_e$, and thus $V(C_1) \cap V_e$ is non-empty. As above we obtain a contradiction by applying \cref{lem:LeanTDsAreKindOfComponentalTight}.
\end{proof}

\begin{lemma} \label{lem:NoLeanTDIntoFiniteParts}
    The graph $G$ from \cref{const:THEGraph} admits no lean \td\ into finite parts.
\end{lemma}

\begin{proof}
    Suppose towards a contradiction that $G$ admits a lean \td\ $(T, \cV)$ into finite parts.
    It follows from \cref{lem:LeanTDKindOfDisplayEndsAndDegrees} that every end $\omega$ of $G$ gives rise to a ray $R_\omega$ in $T$ which arises from no other end of $G$. Moreover, we have $\liminf_{e \in R_\omega} V_e = \Dom(\omega) = \emptyset$, as locally finite graphs have no dominating vertices. So since $(T, \cV)$ is lean and hence strongly linked, \cref{lem:linkedandliminfareonlydominatingverticesimpliesdisplayenddegrees} implies that $\liminf_{e \in R_\omega} |V_e| = \Delta(\omega)= \deg(\omega)$ (because $G$ is locally finite). Thus, by \cref{lem:NoStronglyLinkedTDDisplayingEndDegrees}, $(T,\cV)$ is not strongly linked, which is a contradiction.
\end{proof}

For the proof of \cref{main:NoLeanTD} it remains to show that the graph from \cref{const:THEGraph} neither admits a lean \td\ with possibly infinite parts.

\begin{lemma} \label{lem:NoLeanTD}
    The graph $G$ from \cref{const:THEGraph} admits no lean \td.
\end{lemma}

\begin{proof}
    Suppose for a contradiction that $G$ has a lean \td\ $(T, \cV)$. Let $T$ be rooted in an arbitrary node.
    We first show that every end $\eps_n$ gives rise to a ray $R_n$ in $T$, that is every bag of $(T, \cV)$ meets every $\eps_n$-ray at most finitely often. For this it suffices to show that every bag meets $G_n$ at most finitely often.
    Let $n \in \N$ be given, and suppose there is a bag $V_t$ that contains infinitely many vertices of $G_n$. Then let $Z_1 \subseteq V_t$ be a set of $2^{n+1}+2^n+3$ vertices of $G_n$, and let $i \in \N$ such that $Z_1$ is contained in the first $i$ columns of $G_n$. Now let $Z_2 \subseteq V_t$ be a set of $2^{n+1}+2^n+3$ vertices of $G_n$ that avoids the first $i$ columns of $G_n$. Since the $i$-th column of $G_n$ has size $2^{n+1}+2^n+2$ and separates $Z_1$ and $Z_2$, this contradicts that $(T, \cV)$ is lean.

    Hence, each $\eps_n$-ray meets every bag of $(T, \cV)$ at most finitely often, which implies that $\eps_n$ gives rise to a ray $R_n$ in $T$. In particular, then, there exists a node $t$ of $R_n$ such that $V_t$ meets every row of $G_n$: If not, then every bag $V_t$ of $R_n$ avoids some row of $G_n$. In fact, since $\eps_n$ gives rise to $R_n$ and thus  every $V_t$ meets every row that meets $V_s$ for some node $s <_T t$ of $R_n$, all $V_t$ with $t \in R_n$ avoid the same row of $G_n$. Since $\eps_n$ gives rise to $R_n$, it follows that this row is contained in $G\strictup e$ for all edges $e$ of $R_n$, which contradicts that $(T, \cV)$ is a \td.
    Hence, there is a node $t_n$ of $R_n$ such that $V_{t_n}$ meets every row of $G_n$.
    Let $X_n \subseteq V_{t_n}$ be a set consisting of precisely one vertex of each row of $G_n$. In particular, $X_n$ has size $2^{n+1}+2^n+2$ and is linked to~$\epsilon_n$.

    Since $G$ admits no lean \td\ into finite parts by \cref{lem:NoLeanTDIntoFiniteParts}, $(T, \cV)$ contains an infinite bag $V_s$. As shown above, $V_s$ contains at most finitely many vertices of each $G_n$. 
    Hence, for every $n \in \N$, the bag $V_s$ contains infinitely many vertices of $G'_n := G'[\{(i,j) \in V(G') : j > n\}] \cup \bigcup_{m > n} G_m$. Let $Y_n \subseteq V_s$ consist of $2^{n+1}+2^n+2$ vertices of $G'_n$.  Since $G$ is locally finite and connected, it follows from the Star-Comb Lemma \cite{bibel}*{Lemma~8.2.2} that there is a comb $C$ in $G$ with teeth in $V_s$. 
    Recall that the comb $C$ is the union of a ray $R$, its spine, and infinitely many disjoint (possibly trivial) paths with precisely their first vertex in $R$ and their last vertex in $V_s$. 
    As $V_s \cap V(G_n)$ is finite for all $n \in \N$, the spine $R$ of $C$ is an $\eps$-ray.

    Since $X_n \subseteq V(G_n)$ and $Y_n \subseteq V(G'_n)$, there are at most $2^{n+1}+n+2$ pairwise disjoint $X_n$--$Y_n$ paths in $G$, as witnessed by $S_n$ (indicated in purple in \cref{fig:CounterexUnrootedLinked2}). 
    As $(T, \cV)$ is lean by assumption, there is an edge $f_n$ on the unique $t_n$--$s$ path in $T$ such that $|V_{f_n}| \leq |S_n|$. 
    Moreover, $V_{f_n}$ separates~$V_{t_n}$ and~$V_{s}$. We claim that~$V_{f_n}$ also separates~$\eps_n$ and~$\eps$.
    Indeed, since~$C$ has teeth in~$V_s$, the component $D$ of $G-{V_{f_n}}$ in which $\eps$ lives meets $V_s \setminus V_{f_n}$. 
    Because $X_n$ is linked to $\eps_n$ and $|V_{f_n}| \leq |S_n| < |X_n|$,
    the component $D_n$ of $G - V_{f_n}$ in which $\eps_n$ lives meets $X_n \setminus V_{f_n} \subseteq V_{t_n} \setminus V_{f_n}$.
    Since $V_{f_n}$ separates $V_{t_n}$ and $V_s$, the components $D$ and $D_n$ are distinct, i.e.\ $V_{f_n}$ separates $\eps$ and $\eps_n$.
    By the definition of $G$, there are $|S_n|$ disjoint $\eps$--$\eps_n$ double rays $R_i^n$ in $G$, and hence the separation induced by $f_n$ distinguishes $\eps$ and $\eps_n$ efficiently. 
    Now \cref{lem:NoTreeDecompEffDistAllEnds} yields the desired contradiction.
\end{proof}

\begin{proof}[Proof of \cref{main:NoLeanTD}]
    The graph $G$ from \cref{const:THEGraph} is planar, locally finite and connected. The assertion thus follows from \cref{lem:NoLeanTD}.
\end{proof}

\section{\texorpdfstring{\cref{main:NoRootedlyLeanTreeDecomp}}{Example 3} -- negative answer to question \ref{item:RootedlyLeanTD}} \label{sec:NoRootedlyLeanTD}

In this section we construct \cref{main:NoRootedlyLeanTreeDecomp}, which we restate here for convenience:

\begin{customex}{\cref{main:NoRootedlyLeanTreeDecomp}}
    There is a countable graph $G$ such that every \td\ of $G$ into finite parts has a bag $V_t$ which violates the property of being lean for $s = t$.
\end{customex}

The graph in this example is essentially the same as \cite{carmesin2019displayingtopends}*{Example 3.7}.

\begin{proof}
    Let $G'$ be the $\N \times \{0,1,2\}$ grid, that is, $V(G') = \{(i,j) \mid i \in \N,\; j \in \{0,1,2\}\}$, and there is an edge in $G'$ between $(i,j)$ and $(i',j')$ whenever $|i-i'| + |j-j'| = 1$. For every $n \in \N_{\geq 1}$, let $U_n := \{(i,1) \mid i \leq n\} \cup \{(n-1, 0), (n, 0)\}$ (see \cref{fig:NoRootedlyLeanTD}). 
    Now the graph $G$ is obtained from $G'$ by making the sets $U_n$ complete. We claim that $G$ is as desired.

    \begin{figure}[ht]
        \centering
        \pdfOrNot{\includegraphics[page=2]{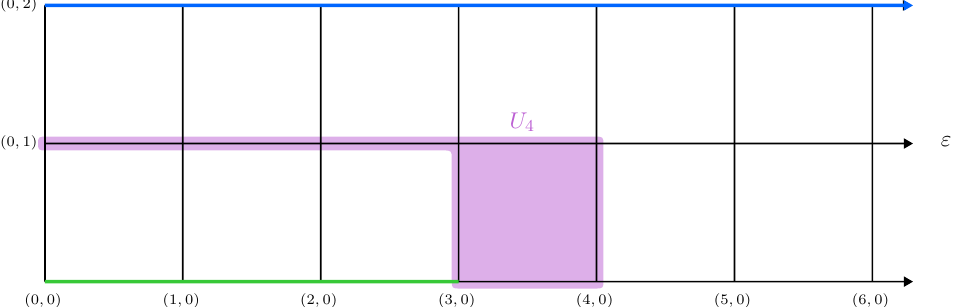}}{\includesvg[width=1.03\columnwidth]{svg/norootedlyleantdexample.svg}}
        \caption{Depicted is the graph $G'$ and the set $U_4$ in purple from \cref{main:NoRootedlyLeanTreeDecomp}. Indicated in blue is the ray $(0,2) (1,2) \dots$ that contains the vertex $w$ and indicated in green is the path $(0,0) (1,0) (2,0) (3,0)$ that contains the vertex $u$ in the case $m = 4$.}
        \label{fig:NoRootedlyLeanTD}
    \end{figure}

    Let $(T, \cV)$ be a rooted \td\ of $G$ into finite parts.
    Let $R \subseteq T$ be the rooted ray arising from the unique end $\eps$ of $G$.
    Since $(T, \cV)$ is a \td, we have $\bigcap_{e \in R} V(G\strictup e) = \emptyset$, and hence there is an edge $e$ of $R$ such that $(0,0), (0,1), (0,2) \in V(G\down e)$. As $\eps$ gives rise to $R$, the ray $(0,0) (1,0) (2,0) \dots$ through the bottom row has a tail in $G \strictup e$; i.e.\ there is $n \in \N$ such that $(n',0) \in V(G\strictup e)$ for all $n' \geq n$. Since $G[U_{n'}]$ is complete, it follows that $U_{n'} \subseteq V(G\up e)$ for all $n' \geq n$. 
    
    Next, observe that there is an edge $e' > e$ of $R$ such that $U_n \subseteq V(G\down e')$.
    Now let $f = t_1t_2$ with $t_1 <_T t_2$ be the $\leq_T$-minimal edge of $R$ such that there exists $m \in \N_{\geq n}$ with $U_m \subseteq V(G\down f)$. Note that $e'$ is a candidate for $f$, and observe that $e < f$.
    Let $m$ be maximal such that $U_m \subseteq V(G\down f)$; in particular, $m \geq n$. 
    To see that this maximum exists, note that the ray $(0,0) (1,0) (2,0) \dots$ has a tail in $G\strictup f$, and hence there is $i \in \N$ such that $(j,0) \in V(G\strictup f)$ for all $j \geq i$. Thus, $m \leq i$.
    
    By the choice of $f$ and $m$, we have $U_{m+1} \subseteq V(G\up f)$, as $U_{m+1}$ is complete. Hence, as $V_f$ separates $G\strictup f$ and $G\strictdown f$, it follows that $U_m \cap U_{m+1} \subseteq V_f \subseteq V_{t_1}$. Moreover, as $(0,2) \in V(G\down e) \subseteq V(G\down f)$ by the choice of $e$, and because $\eps$ gives rise to $R$, the ray $(0,2) (1,2) (2,2) \dots$ meets $V_f \subseteq V_{t_1}$ in a vertex~$w$. Similarly, the path $(0,0) (1,0) \dots (m-1, 0)$ meets $V_{f'} \subseteq V_{t_1}$ in a vertex $u$ where~$f'$ is the unique down-edge at $t_1$. Indeed, we have $(0,0) \in G\down f'$ by the choice of $e < f$ and $(m-1,0) \in G\up f'$: if $(m-1, 0)$ was contained in $G\strictdown f'$, then $U_m \subseteq V(G\down f')$ since $U_m$ is complete in $G$, which contradicts the $\leq_T$-minimal choice of $f$.

    We now let $s:= t_1 =: t$ and define $Z_1 := (U_m \cap U_{m+1}) \cup \{w\}$ and $Z_2 := (U_m \cap U_{m+1}) \cup \{u\}$. By construction, we have $|Z_1| = |Z_2| = m + 3$ and $Z_1, Z_2 \subseteq V_{t_1}$. 
    Hence, $(T, \cV)$ violates the property of being lean for $s = t_1 = t$ since $U_m \cap U_{m+1}$ separates $w$ and $u$ and hence witnesses that $G$ contains at most $m + 2$ disjoint $Z_1$--$Z_2$ paths.
\end{proof}

\section{Upwards disjointness of adhesion sets} \label{sec:IncDisj}

As mentioned in the introduction, we show in \cite{LinkedTDInfGraphs}*{Theorem~1'} a more detailed version of \cref{thm:LinkedTightCompTreeDecomp} which, among others, yields that the adhesion sets of the \td\ intersect `not more than necessary': 

\begin{theorem} \label{thm:LinkedTightCompTreeDecompIncDisj}
    Every graph~$G$ of finite tree-width admits a rooted \td\ $(T, \cV)$ into finite parts that is linked, tight and componental.
    Moreover, we may assume that
    \begin{enumerate}[label=\rm{(\arabic*)}]
        \item \label{itm:MainTechnical:IncDisj} for every $e <_T e' \in E(T)$ with $|V_e| \leq |V_{e'}|$, each vertex of $V_e \cap V_{e'}$ either dominates some end of~$G$ that lives in~$G \up e'$, or is contained in a critical vertex\footnote{A set $X$ of vertices of $G$ is \defn{critical} if infinitely many components of $G-X$ have neighbourhood $X$ in $G$.} set of $G$ that is included in $G\up e'$.
    \end{enumerate}
\end{theorem}

Halin \cite{halin1975chainlike}*{Theorem~2} showed that every locally finite, connected graph has a linked ray-decomposition\footnote{A \defn{ray-decomposition} is a \td\ whose decomposition tree is a ray.} into finite parts with disjoint adhesion sets. He used this result in \cite{halin1977systeme}*{Satz~10} to establish \cref{thm:LinkedTightCompTreeDecompIncDisj} for locally finite graphs with at most two ends, replacing \cref{itm:MainTechnical:IncDisj} by the stronger condition of having disjoint adhesion sets.
In light of this, we discuss here that \cref{itm:MainTechnical:IncDisj} describes how close one may come to having `disjoint adhesion sets' in the general case.

If~$G$ is not locally finite, we generally cannot require the \td\ $(T, \cV)$ in \cref{thm:LinkedTightCompTreeDecomp} to have disjoint adhesion sets while having finite parts, as every dominating vertex of an end $\eps$ will be eventually contained in all adhesion sets along the ray of $T$ which arises from $\eps$.
Moreover, as every critical vertex set has to lie in an adhesion set of any \td\ into finite parts, and since the \td\ is linked, one can easily check that the adhesion sets also intersect in critical vertex sets.
Thus one might hope to obtain a \td\ as in \cref{thm:LinkedTightCompTreeDecomp} that satisfies the following condition:
\begin{enumerate}[label=\rm{(\arabic*')}]
    \item \label{itm:MainTechnical:IncDisjStronger} for every $e <_T e' \in E(T)$ each vertex of $V_e \cap V_{e'}$ either dominates some end of~$G$ that lives in~$G \up e'$, or is contained in a critical vertex set of $G$ that is included in $G\up e'$. 
\end{enumerate}

But \cref{itm:MainTechnical:IncDisj} allows for more than \cref{itm:MainTechnical:IncDisjStronger}: If $e <_T e' \in T$ and $|V_e| > |V_{e'}|$, then $V_e$ and $V_{e'}$ are allowed to intersect also in vertices that do not dominate an end and that are not contained in a critical vertex set. The following example shows that allowing this is in fact necessary. 
It presents a locally finite graph that does not admit a \td\ $(T,\cV)$ as in \cref{thm:LinkedTightCompTreeDecomp} with \cref{itm:MainTechnical:IncDisjStronger}, the stronger version of \cref{itm:MainTechnical:IncDisj}.
As locally finite graphs do not have any dominating vertices and critical vertex sets, \cref{itm:MainTechnical:IncDisjStronger} boils down to the property that the \td\ $(T, \cV)$ has \defn{upwards disjoint adhesion sets}, that is $V_e \cap V_{e'} = \emptyset$ for every $e <_T e' \in E(T)$.

\begin{example} \label{ex:NoTDWithUpwardsDisjointAdhesionSets}
    There is a locally finite connected graph $G$ which does not admit a linked, tight, componental rooted \td\ $(T, \cV)$ into finite parts with upwards disjoint adhesion sets, i.e.\ one which satisfies \cref{itm:MainTechnical:IncDisjStronger}.
\end{example}

\begin{proof}
Let $c_{00}(\N)$ denote the set of all sequences with values in $\N$ that are eventually zero.
    Let $c' \subseteq c_{00}(\N)$ be the set of all the sequences $\cS = (s_n)_{n \in \N}$ for which there exists $N \in \N$ such that $s_i \geq 1$ for all $i < N$ and $s_i = 0$ for all $i \geq N$.\footnote{We restrict to the set $c'$ instead of $c_{00}(\N)$ to ensure that $G$ is connected.}
    Let $G$ be the graph depicted in \cref{fig:CounterexUpDisj1}, that is the graph on the vertex set 
    \[
    V(G) := \{((s_n)_{n \in \N}, i, j) \in c' \times \N \times \{1,2,3\}\}
    \]
    and with edges between $(\cS, i, j)$ and $(\cS, i', j')$ whenever $|i-i'| + |j-j'| = 1$ and, for $\cS = (s_0, \dots, s_{n-1}, 0, \dots)$ and $\cS' = (s'_0, \dots, s'_{n-1}, s'_{n}, 0, \dots)$, with edges between $(\cS, i-1, 3)$ and $(\cS', 0, j')$ and between $(\cS, i, 3)$ and $(\cS', 0, j')$ whenever $s'_{n} = i \geq 1$ and $s_k = s'_k$ for all $k < n$.  Intuitively, a sequence such as $\cS = (3,2,3,0,0,\ldots)$ encodes  the specific $\N \times 3$ grid that one encounters when following along the (blue) $(0,0,0,\ldots)$ grid, then turning after 3 steps into the (green) $(3,0,0,\ldots)$ grid, then turning after another 2 steps into the (black) $(3,2,0,0,\ldots)$ grid, and finally turning after another 3 steps into the $(3,2,3,0,0,\ldots)$ grid (not depicted).
    Note that $G$ is locally finite and connected.

    \begin{figure}
        \centering
        \pdfOrNot{\includegraphics[page=2]{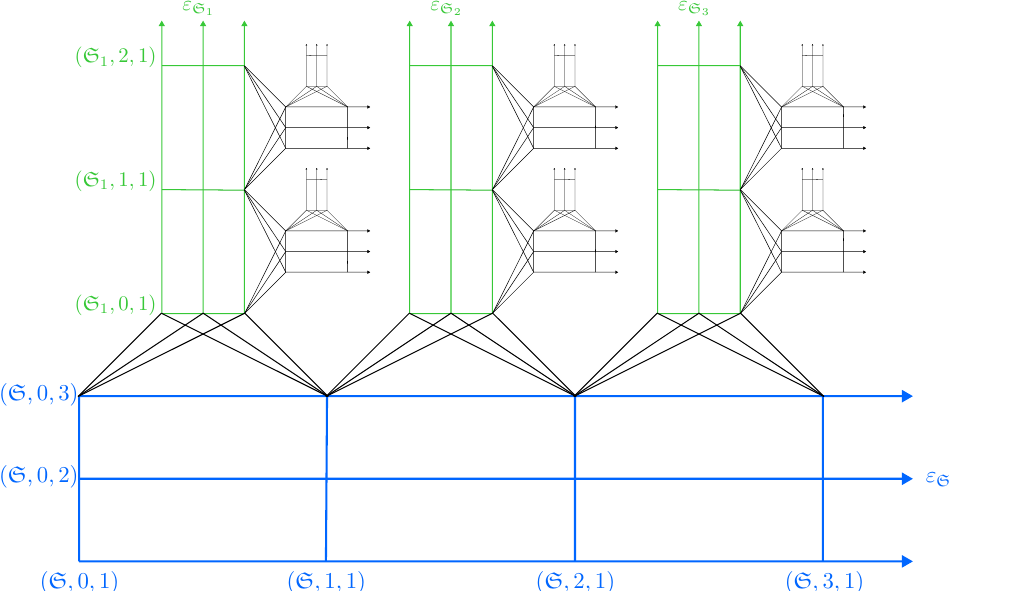}}{\includesvg[width=1.09\columnwidth]{svg/upwardsdisjointexample_coloured.svg}}
        \caption{Depicted is the graph $G$ from \cref{ex:NoTDWithUpwardsDisjointAdhesionSets}. The blue subgraph is induced by the vertices $(\frakS, i, j)$ of $G$ with $\frakS = (0, 0, 0, \dots)$. The green subgraphs are induced by the vertices $(\frakS_x,i,j)$ of $G$ with $\frakS_x = (x,0,\dots)$ for $x \in \N_{\geq 1}$, respectively.}
        \label{fig:CounterexUpDisj1}
    \end{figure}

    By \cref{lem:LinkedTightCompImpliesDisplayingEndStructure} every linked, tight, componental rooted \td\ of $G$ into finite parts displays all the ends of $G$ and their (combined) degrees.
    Thus, it suffices to show that~$G$ has no rooted \td\ with upwards disjoint adhesion sets which displays all its ends and their (combined) degrees.
    Let $(T, \cV)$ be a \td\ of~$G$ which displays all ends of~$G$ and their (combined) degrees. 
    Consider the rays $R_{\cS, j} = \{(\cS, i, j) \mid i \in \N\}$ for all $\cS \in c'$ with $j \in \{1,2,3\}$. 
    Then for every fixed $\cS \in c'$ the rays $R_{\cS, 1}, R_{\cS,2}, R_{\cS,3}$ all belong to the same end $\eps_\cS$ of $G$; and these ends $\eps_\cS$ are pairwise distinct and have (combined) degree $3$.
    Since $(T, \cV)$ displays all ends of~$G$ and their (combined) degrees, there exist for each $\eps_\cS$ infinitely many edges $e$ of $T$ such that $V_e$ has size three and meets every ray $R_{\cS,j}$; we fix for each end $\eps_\cS$ one such edge $e_\cS$. Let $v_\cS = (\cS, i_\cS, 3)$ be the (unique) vertex in $V_{e_\cS} \cap V(R_{\cS,3})$ (see \cref{fig:CounterexUpDisj2}). 

    Set $\cS_0 := (0, 0, \dots)$ and $\cS_n := (s_0, \dots, s_{n-1}, i_{\cS_{n-1}}+1, 0, \dots)$ where $\cS_{n-1} = (s_0, \dots, s_{n-1}, 0, \dots)$.
    Then the $v_{\cS_n}$ define a (unique) end~$\eps$ of $G$, in that every ray that meets all the $v_{\cS_n}$ belongs to the same end~$\eps$. This end has degree $2$ as witnessed by the sets $S_n := \{(\cS_n, i_{\cS_n}, 3), (\cS_n, i_{\cS_n} + 1, 3)\}$ (see \cref{fig:CounterexUpDisj2}). 

    \begin{figure}
        \centering
        \pdfOrNot{\includegraphics[page=2]{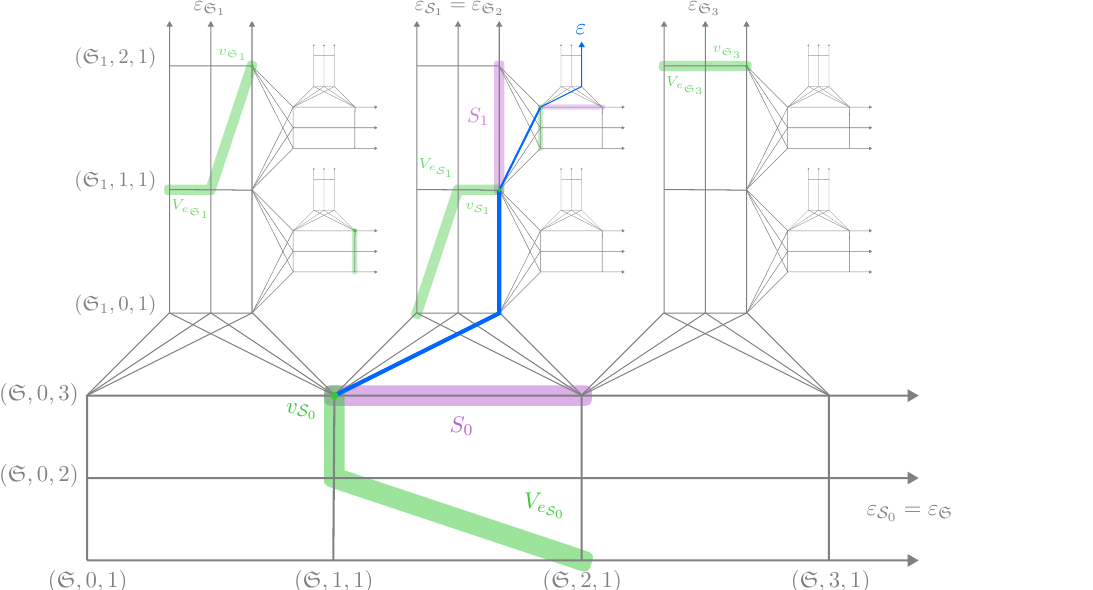}}{\includesvg[width=1.15\columnwidth]{svg/upwardsdisjointexample_ray.svg}}
        \caption{A ray in $G$ that meets all the sets $V_{e_{\cS_n}}$. The end to which it belongs has degree 2.}
        \label{fig:CounterexUpDisj2}
    \end{figure}
    
    Now since $(T, \cV)$ displays the (combined) degree of $\eps$, and because the sets $S_n$ are the only separators witnessing that $\eps$ has degree $2$, there is an edge $e \in T$ with $V_e = S_n$ for some $n \in \N$. But since every such~$S_n$ meets the set $V_{e_{\cS_n}}$, the \td\ $(T, \cV)$ does not have upwards disjoint separators.
\end{proof}

This example might explain why Halin never extended his result \cite{halin1977systeme}*{Satz~10} mentioned above to graphs with more than two ends: His precursor notion to a tree-decomposition, namely the quasi-trees and pseudo-trees discussed in \cite{halin1991tree}, required upwards disjoint separators, and hence could not possibly capture the types of locally finite graphs described in \cref{ex:NoTDWithUpwardsDisjointAdhesionSets}.

\section*{Acknowledgements}

The second named author gratefully acknowledges support by doctoral scholarships of the Studienstiftung des deutschen Volkes and the Cusanuswerk -- Bisch\"{o}fliche Studienf\"{o}rderung.
The third named author gratefully acknowledges support by a doctoral scholarship of the Studienstiftung des deutschen Volkes.

\bibliographystyle{amsplain}
\arXivOrNot{\bibliography{collectivearXiv.bib}}{\bibliography{collective.bib}}

\end{document}

%% file: Preamble.tex
\makeatletter
\let\origsection=\section \def\section{\@ifstar{\origsection*}{\mysection}} 
\def\mysection{\@startsection{section}{1}\z@{.7\linespacing\@plus\linespacing}{.5\linespacing}{\normalfont\scshape\centering\S}}
\makeatother  

\usepackage{amsmath}
\usepackage{amssymb}
\usepackage{amsthm}
\usepackage{mathtools}
\usepackage{letterswitharrows}
\usepackage{tikz-cd}
\usepackage[only,llbracket,rrbracket]{stmaryrd}
\usepackage[linesnumbered,ruled,vlined]{algorithm2e}
\usepackage{enumitem}
\setenumerate{label={\normalfont (\roman*)}}

\usepackage[utf8]{inputenc}
\usepackage[T1]{fontenc}
\usepackage{lmodern}
\usepackage[babel]{microtype}
\usepackage[english]{babel}
\usepackage{relsize}

\usepackage{graphicx}
\usepackage{subcaption}
\usepackage{import}

\usepackage{svg}

\usepackage{geometry}
\usepackage{xcolor} 
\colorlet{darkishRed}{red!80!black}
\colorlet{darkishBlue}{blue!60!black}
\colorlet{darkishGreen}{green!60!black}
\usepackage{hyperref}
\hypersetup{
	colorlinks,
	linkcolor={red!60!black},
	citecolor={green!60!black},
	urlcolor={blue!60!black},
}
\usepackage[abbrev, msc-links]{amsrefs}
\usepackage[nameinlink, capitalise, noabbrev]{cleveref}
\crefformat{enumi}{#2#1#3}
\crefformat{equation}{#2(#1)#3}
\crefname{mainresult}{Theorem}{Theorems}
\usepackage{doi}

\renewcommand{\PrintDOI}[1]{\doi{#1}}

\let\setminus=\smallsetminus

\renewcommand{\leq}{\leqslant}
\renewcommand{\geq}{\geqslant}

\newtheorem{theorem}{Theorem}[section]

\newtheorem{lemma}[theorem]{Lemma}

\newtheorem{mainresult}{Theorem} 

\newtheorem{innercustomex}{}

\newenvironment{customex}[1]
  {\renewcommand\theinnercustomex{#1}\innercustomex}
  {\endinnercustomex}
\newcounter{claimcounter}[theorem]
\newtheorem*{claim*}{Claim}

\newcounter{subclaimcounter}[claimcounter]
\newtheorem*{subclaim*}{Subclaim}

\newtheorem{mainexample}[mainresult]{Example}
\crefname{mainexample}{Example}{Examples}
\newtheorem{example}[theorem]{Example}
\crefname{example}{Example}{Examples}

\theoremstyle{definition}

\newtheorem{construction}[theorem]{Construction}

\crefname{routine}{Routine}{Routines}

\crefname{subroutine}{Subroutine}{Subroutines}

\crefname{subsubroutine}{Subsubroutine}{Subsubroutines}

\crefname{step}{Step}{Steps}
\theoremstyle{remark}

\newcommand{\red}[1]{{\color{red}{#1}}}
\newcommand{\blue}[1]{{\color{blue}{#1}}}

\newcommand{\COMMENT}[1]{{}}

\let\eps=\varepsilon
\let\epsilon=\varepsilon
\let\theta=\vartheta
\let\rho=\varrho
\let\phi=\varphi
\def\N{\mathbb N}

\makeatletter

\def\calCommandfactory#1{%
  \expandafter\def\csname c#1\endcsname{\mathcal{#1}}}
\def\frakCommandfactory#1{%
  \expandafter\def\csname frak#1\endcsname{\mathfrak{#1}}}
   
\newcounter{ctr}
\loop
  \stepcounter{ctr}
  \edef\X{\@Alph\c@ctr}
  \expandafter\calCommandfactory\X
  \expandafter\frakCommandfactory\X
\ifnum\thectr<26
\repeat

\usepackage{etoolbox}

\newbool{arXiv}
\booltrue{arXiv} 

\newcommand{\arXivOrNot}[2]{\ifbool{arXiv}{{#1}}{{#2}}}

\newbool{pdfBool}
\booltrue{pdfBool} 

\newcommand{\pdfOrNot}[2]{\ifbool{pdfBool}{{#1}}{{#2}}}

\newbool{TrackBool}
\booltrue{TrackBool} 

\newcommand{\change}[2]{\ifbool{TrackBool}{\blue{#1} \red{#2}}{{#1}}}

%% file: specificPreamble.tex
\newcommand{\td}{tree-decom\-posi\-tion}

\newcommand{\rd}{ray-decom\-posi\-tion}
\newcommand{\down}{{\downarrow}}
\newcommand{\up}{{\uparrow}}
\newcommand{\strictdown}{ \mathring{\downarrow} }
\newcommand{\strictup}{ \mathring{\uparrow} }

\DeclareMathOperator{\Dom}{Dom}

\newcounter{mylabelcounter}

\makeatletter
\newcommand{\labelText}[2]{%
#1\refstepcounter{mylabelcounter}%
\immediate\write\@auxout{%
  \string\newlabel{#2}{{1}{\thepage}{{\unexpanded{#1}}}{mylabelcounter.\number\value{mylabelcounter}}{}}%
}%
}

\makeatletter
\newcommand\footnoteref[1]{\protected@xdef\@thefnmark{\ref{#1}}\@footnotemark}
\makeatother

\renewcommand\boxtimes{\mathbin{\text{\scalebox{.84}{$\square$}}}}